\documentclass[12pt]{article}

\usepackage{times}

\usepackage[a4paper,text={128mm,185mm}, centering]{geometry}

\usepackage{changepage}
\usepackage{titlesec}

\titleformat{\section}[hang]%
{\bfseries\large}{\thesection.}{1ex}{}%

\titleformat{\subsection}[hang]%
{\bfseries}{\thesubsection}{1ex}{}%

\usepackage{amsthm}
\usepackage{fancyhdr}
\usepackage{xcolor}
\theoremstyle{plain}
\newtheorem{theorem}{Theorem}[section]

\newtheorem{proposition}[theorem]{Proposition}
\newtheorem{corollary}[theorem]{Corollary}
\newtheorem{definition}[theorem]{Definition} 

\theoremstyle{definition}
\newtheorem{example}[theorem]{Example}
\newtheorem{remark}[theorem]{Remark}

\usepackage[matrix,arrow,curve,cmtip]{xy}
\usepackage{amssymb}
\usepackage{latexsym}
\usepackage{graphicx}
\usepackage{amsmath}

\usepackage{mathrsfs} 
\usepackage{xypic}

\title{\vskip 5pt  \bf  A simplified categorical approach to several Galois theories}
\author{\itshape{D. Bl\'azquez-Sanz, C. A. Mar\'{\i}n Arango \& J. F. Ruiz Castrillon}}
\date{}

\begin{document}
\maketitle

%Nota bene:  the next two commands erase pages numbers (the right page numbers will be add by the editors on the pdf) 
%\cfoot{}
\thispagestyle{empty}
\vskip 25pt
% Abstract in French and in English, followed by Keywords and MSC:
\begin{adjustwidth}{0.5cm}{0.5cm}
{\small
{\bf R\'esum\'e.} Nous \'etudions le concept de structure de Galois et epimorphisme de Galois dans un contexte g\'en\'eral. Notamment, une structure de Galois pour un \'epimorphisme $\pi\colon M \rightarrow B$ dans une cat\'egorie $\mathcal C$ %. Il s'agit de 
est l'action d'un groupe objet qui munit $M$  d'une structure d'espace homog\`ene dans la cat\'egorie relative ${\mathcal C}_B$. %Nous verrons que ces concepts g\'en\'eraux s'appliquent pour les recouvrements ; les extensions finies des corps ; les extensions fortement normales des corps différentiels ; etc. De plus, nous allons explorer des structures de Galois dans la cat\'egorie des vari\'et\'es feuillet\'ees, en arrivant \`a un \'equivalent purement g\'eom\'etrique et lisse de la th\'eorie de Galois diff\'erentielle.
\\
{\bf Abstract.} We discuss the concept of Galois structure and Galois epimorphism in a general setting. Namely, a Galois structure for an epimorphism $\pi\colon M\to B$ in some category ${\mathcal C}$ is the action of a group object that gives to $M$ the structure of principal homogeneous space in the relative category ${\mathcal C}_B$. %We see that this general setting applies to coverings, finite field extensions, strongly normal extensions of differential fields, etc. We also explore Galois structures in the category of foliated manifolds, arriving to a purely geometric and smooth counterpart of differential Galois theory. \\
}
\end{adjustwidth}

% Here starts the text:.
\section{Introduction}

From its very starting point in the theory of polynomial equations with one variable \cite{Galois}, Galois theory proposes a systematic use of the principal homogeneous structure of the space of solutions of an equation. This idea was systematically applied by E. Vessiot \cite{Vessiot1904} in his general approach to differential Galois theory. Today there are several Galois theories, with different domains of application. 

\smallskip

It is clear that there is some common mathematical core within all these theories. This is usually explained through analogy. Most texts dedicated to several Galois theories develop them separately, establish some bridges, and point out these analogies between them, as in the book of R. and D. Douady  \cite{Douady2005}. 

\smallskip

There is a categorical approach to Galois theory initiated by Grothendieck (\cite{SGA1}, see \cite{Dubuc2000} for a more accessible exposition) and continued in \cite{Artin1972}  (see also \cite{JOHNSTONE1977}  ). This theory is further developed by Dubuc \cite{DUBUC2003}  and culminated by Joyal-Tierney \cite{Joyal1984}. A different approach to Galois theory is considered by G. Janelidze and F. Bourceux (\cite{Janelidze1990}, see also \cite{Borceux}, chapter 5). This categorical Galois theory does not cover some natural incarnations of Galois theory, as differential Galois theory \cite{SingerVanderPut2003}. The main difference between Grothedieck approach and ours is the following: we do not see the Galois group as a set-theoretical group acting on an object but as a group object of the category. This line of thinking is inspired by some facts of differential Galois theory. For instance, the Galois group of a strongly normal extension \cite{Kolchin1973} is an algebraic group defined over the constants, which can be seen as a particular kind a group object in the category of differential algebraic varieties. Some years ago A. Pillay generalized E. Kolchin's theory of strongly normal extensions \cite{Pillay1998}. A generalized strongly normal extension is a differential field extension whose group of automorphisms admits a natural structure of differential algebraic group, that is, a group object in the category of differential algebraic varieties.

\smallskip

%The ideas we present here were already available, sometimes in hidden form, in specific contexts: automorphic systems \cite{Vessiot1904}, Galoisian categories and faithfully flat descent \cite{SGA1}, the presentation of classical Galois theory as in the book of C. Sancho and P. Sancho \cite{SanchoSancho2013}, Galois theory of rings \cite{Villamayor1966,Magid1971}, Galois theory for inseparable extensions \cite{Chase1976},
%theory of $V$-primitive extensions \cite{Kolchin1973}, geometric characterization of strongly normal extensions in \cite{Kovacic2003,Kovacic2006}, the principal homogeneous structure over a definable group in generalized strongly normal extensions of A. Pillay \cite{Pillay1998}, etc.
%\smallskip

Our framework also explains how some Galois theories are naturally extended. Most of them allow Galois structures (Definition \ref{df:Galois_st}) with Galois groups in some specific class of group objects in a category. By modifying the category, or by extending the class of possible Galois groups we obtain different extensions of Galois theory. For instance, classical Galois theory extends to Hopf-Galois theory by allowing a broader class of group objects. 
%Something similar happens for Picard-Vessiot theory, strongly normal extensions and generalized strongly normal extensions.  

\smallskip

We give some examples of how the proposed general definitions apply to the cases of classical Galois theory (algebraic and topological), and differential Galois theory. Then we explore the category of foliated smooth manifolds. Epimorphims in such category are partial Ehresmann conections. When examining Galois structures there naturally appear $G$-invariant connections. This is not surprising, $G$-invariant connections were in fact introduced in the context of Galois theory by E. Vessiot in the beggining to 20th century: they are the so-called automorphic systems appearing in \cite{Vessiot1904}. We prove uniqueness of the Galois group for the irreducible case, Theorem \ref{smooth}. Finally we compare the real smooth and the complex algebraic cases.

%\smallskip
%As our examples cover algebra, topology, differential algebra, and differential geometry, the text is not self contained and we do not include all definitions. We use the standard notation in each context, and we also suggest some reference where it is indispensable.

\section{General definitions}

\subsection{Split of groupoid actions}

Let us consider ${\mathcal C}$ a category with binary products, kernels of pairs of morphisms, and a final object $\{\star\}$. Thus, there are also fibered products (pullbacks) as well as finite limits. We may define group objects and groupoid objects in ${\mathcal C}$.

\medskip

Let $G$ be a group object in ${\mathcal C}$. For each object $X$, the set $G(X) =  {\rm Hom}(X,G)$ of $X$-elements of $G$ is a group. An action of $G$ in an object $M$ is a morphism,
$$\alpha\colon G\times M \to M,$$
satisfying $\alpha\circ(\mu\times {\rm Id}_M) = \mu\circ({\rm Id}_{G}\times\alpha)$
and $\alpha\circ ((e_{G}\circ \pi_M)\times {\rm Id}_M) = {\rm Id}_M$.\footnote{
	Where $e_G$ represent the identity $e_G\colon \{\star\}\to G$ and $\pi_M$ represents the unique morphism $\pi_M\colon M\to \{\star\}$.
}
The action $\alpha$ induces a group morphism $\alpha \colon G(\{\star\})\to {\rm Aut}(M)$, $ g\mapsto \alpha\circ \langle g\circ \pi_M, {\rm Id}_M\rangle$. 

\smallskip

From the action $\alpha$ we can form the \emph{action groupoid} $G\ltimes M \rightrightarrows M$,
with objects of objects $M$ and object of arrows $G\ltimes M$. The source map is the projection $\pi_2$ onto the second factor $M$, and the target map is $\alpha$. In terms of sets and elements, we have:
$$s(g,x) = x, \quad t(g,x) = \alpha(g,x), \quad  (h,gx)\circ(g,x) = (hg,x).$$

\begin{definition}
	We say that a groupoid object $\mathcal G\rightrightarrows M$ splits in $\mathcal C$ if there is an
	action $\alpha \colon G\times M \to M$ an action of a group object and a groupoid isomorphism $\varphi\colon G \ltimes M \xrightarrow{\sim} \mathcal G$. In such a case, we say that $G$ is a splitting group, $\alpha$ is a splitting action and $\varphi$ is a splitting morphism for $\mathcal G$ in ${\mathcal C}.$
\end{definition}

\begin{example} Let us remark that it is not in general possible to recover the group $G$ from the action groupoid $G\ltimes M$. For instance, in the category 
	of sets, let us consider two free and transitive actions of $\mathbb Z_4$ and $\mathbb K_4$ in a set $X = \{p_1,p_2,p_3,p_4\}$ of four elements. Since the actions are free and transitive we have that the action groupoid is, in both cases, the total equivalence realtion $X\times X$. Therefore we have groupoid isomorphism:	
		$$\xymatrix{ 
		\mathbb K_4\ltimes X \ar[r]^-{\sim} & X\times X & \mathbb Z_4 \ltimes X \ar[l]_-{\sim} \\
		(\sigma,p) \ar[r] & (p,\sigma\cdot p),(p,\tau\cdot p) & (\tau,p) \ar[l] 
	}$$
	Thus, a split groupoid object may have different realizations as an action groupoid.
\end{example}

\subsection{Normal epimorphisms}

Let us recall that an action of a group (set) $G$ in an object $X$ is a group morphism
$\phi\colon G\to {\rm Aut}(X)$. We say that $q\colon X\to Y$ is a \emph{categorical quotient} of the action of $G$ in $X$ if:
\begin{enumerate}
	\item For all $g\in G$, $q\circ \phi(g) = q$. In other words, $q$ is $G$-invariant.
	\item For all morphisms $f\colon X\to Z$ such that for all $g\in G$ $f\circ \phi(g) = f$ (i.e. $f$ is $G$-invariant) there exists a unique $\bar f\colon Y\to Z$ such that $\bar f\circ q = f$.
\end{enumerate}

Categorical quotients are epimorphisms and are unique up to isomorphims. Let us consider $\pi\colon M\to B$ an epimorphism in ${\mathcal C}$. The group group ${\rm Aut}_B(M)$ acts on $M$.

\begin{definition}
We say that $\pi$ is normal if it is the categorical quotient of $M$ by the action of the group (set) ${\rm Aut}_B(M)$.
\end{definition}

Some categorical approaches to Galois theory rely in the notion of strict epimorphism (\cite[I.10.2]{Artin1972} see also \cite[Def. 5.1.6]{Kashiwara2005}). 

\begin{definition}
Let $\pi\colon M\to B$ be an epimorphism.   
\begin{itemize}
    \item[(a)] A morphism $f\colon M\to Z$ is $\pi$-compatible if for any pair of morphisms $x,y\colon X\rightrightarrows M$ such that $\pi\circ x = \pi \circ y$ also $f\circ x = f\circ z$.
    \item[(b)] $\pi$ is a strict epimorphism if for any $\pi$-compatible $f$ there is a unique $\bar f\colon B\to Z$ such that $f = \bar f\circ \pi$. 
\end{itemize}
\end{definition}

\begin{proposition} Let $\pi\colon M\to B$ be an epimorphism in a category $\mathcal C$.
	\begin{enumerate}
		\item[(a)] If $\pi$ is normal then it is strict.
		\item[(b)]  Assume that any arrow with codomain $M$ is invertible. Then, if $\pi$ is strict $\pi$ is normal.
	\end{enumerate}

\end{proposition}

\begin{proof}
Let us consider an object $Z$ an the composition map 
$$\pi^*\colon {\rm Hom}(B,Z)\to {\rm Hom}(M,Z).$$ 
The image of $\pi^*$ consists of $\pi$-compatible morphisms. Moreover, let us assume that $f\colon M\to Z$ is $\pi$-compatible. Then, for any $\sigma\in {\rm Aut}_B(M)$ we have $\pi \circ \sigma = \pi \circ {\rm Id}_M$ and therefore $f\circ \sigma = f$. It means that $\pi$-compatible morphisms are invariant under the action of ${\rm Aut}_B(M)$. In general we have a chain,\footnote{Here ${\rm Hom}(M,Z)^{{\rm Aut}_B(M)}$ stands for the set of ${\rm Aut}_B(M)$-invariant morphisms in ${\rm Hom}(M,Z)^{{\rm Aut}_B(M)}$.}
$$\pi^*({\rm Hom}(B,Z)) \subseteq \{\pi\mbox{-compatible morphisms}\} \subseteq {\rm Hom}(M,Z)^{{\rm Aut}_B(M)}.$$
Let us note the following: 
\begin{enumerate}
	\item[(i)] The epimorphism $\pi$ is normal if and only if for any $Z$ we have the equality between the first and third members of the chain.
	\item[(ii)] The epimorphism $\pi$ is strict if and only if for any $Z$ we have the equality between the first and second members of the chain.
\end{enumerate}
(a) Assume $\pi$ normal. Then the three members of the above chain coincide. In particular, any $\pi$-compatible morphism factorizes.

\noindent (b) Assume that $\pi$ is strict. We need to prove that any  ${\rm Aut}_B(M)$ invariant morphism $f\colon M\to Z$ is $\pi$-compatible. Let $a,b\colon X\rightrightarrows M$ be a pair of morphisms such that $\pi\circ a = \pi\circ b$. Since $f$ is ${\rm Aut}_B(M)$ invariant we have $f = f \circ (b\circ a^{-1})$ and from this $f\circ a = f\circ b$. Hence $f$ is $\pi$-compatible. 
\end{proof}

\begin{remark}
Let us recall that there are different of regular and effective epimorphism.
\begin{enumerate}
    \item[(a)] An epimorphism $q\colon Y\to X$ is said to be regular if it is the coequalizer of a pair of morphisms $Z\rightrightarrows Y \to Z.$
    \item[(b)] An epimorphism $q\colon Y\to X$ is said to be effective if it has a kernel pair and it is the coequalizer of a congruence of its kernel pair
    ${\rm KP}_q\rightrightarrows Y \to X$.
\end{enumerate}
In a general category we have:
$$\mbox{ effective } \Longrightarrow \mbox{ regular } \Longrightarrow \mbox{ strict }.$$
Moreover, in a category with pullbacks it is known that strict epimorphisms are effective. Therefore a \emph{normal} epimorphism in a category with pullbacks is effective. 
\end{remark}

\subsection{Galois structures}

The kernel pair of $\pi$, ${\rm KP}_\pi = M\times_B M \rightrightarrows M$, is a congruence (equivalence relation) in $M$, and therefore a grupoid object in $\mathcal C$. We set the source ($s$) and target ($t$) maps to be the first and second projection respectively. 
%Since $\mathcal C$ is a category with pullbacks, the notions of strict, regular and effective epimorphisms coincide. Therefore, another way of saying that $\pi$ is a strict epimorphism is to say that $B$ is the categorical quotient of $M$ by such a relation. In particular $M\times_B M$ can be seen as a groupoid acting on $M$. This is the key fact of Galois theory that yields the following definition. 
%\begin{definition}
%	Let $\pi\colon M \to B$ be an epimorphism. We call Galois groupoid of $\pi$ to the groupoid ${\mathcal Gal}_{\pi} = M \times_B M \rightrightarrows M$ where the source and target morphisms are the natural projections $\pi_1$ and $\pi_2$ respectively.  
%\end{definition}
It represents the endomorphisms of $M$ over $B$ in the following sense: let ${\rm KP}_{\pi}(M)$ be the set of sections of the source map $(s)$; the composition with the target map yields a bijection.
$$\xymatrix{ & {\rm KP}_{\pi} \ar[ld]^-{s} \ar[rd]_-{t} &  & {\rm KP}_{\pi}(M)\ar[r]^-{\sim} & {\rm End}_B(M) \\  M\ar[rr]^-{t\circ\sigma} \ar@/^1.2pc/[ur]^-{\sigma} & & M & \sigma \ar[r]^-{\sim} & t\circ\sigma }$$

Let us consider an splitting action $\alpha\colon G\times  M \to M$ of ${\rm KP}_{\pi}$. 
The splitting isomorphism is necessarily 
$$\langle \pi_2,\alpha \rangle  \colon G\ltimes M 
\xrightarrow{\sim} {\rm KP}_{\pi}\quad (g,x)\mapsto (x,\alpha(g,x)),$$
which is completely determined by $\alpha$. In other words, an splitting action of ${\rm KP}_\pi$ is an action that gives $\pi\colon M\to B$ the structure of \emph{principal homogeneous space} modeled over $G\times B \to B$ in the relative category ${\mathcal C}_B$ of arrows over $B$.

\smallskip

Let us note that the splitting action $\alpha$ induces
a bijection between $G(M)$ and ${\rm KP}_\pi(M)$ and therefore a bijection,
$$G(M) \xrightarrow{\sim} {\rm End}_B(M),
\quad g \mapsto \alpha_g = \alpha\circ \langle g, {\rm Id}_M\rangle.$$ 
However, such bijection is not compatible with the composition. We have
$$\alpha_{g}\circ \alpha_h = 
\alpha\circ \langle g , {\rm Id}_M \rangle \circ \alpha \circ \langle h , {\rm Id}_M\rangle  = \alpha\circ \langle g\circ \alpha_h , \alpha_h \rangle$$
on the other hand,
$$\alpha_{gh} = \alpha\circ \langle g , \alpha_h \rangle.$$
It follows that, if $g = g\circ \alpha_h$ then, $\alpha_{gh} = \alpha_g\circ \alpha_h$. We see that this is satisfied if $g \in G(B)$, given that $\alpha_h\in {\rm End}_B(M)$ induces the identity in $B$. For \emph{normal} epimorphisms this condition is optimal, as 
$G(M)^{{\rm Aut}_B(M)} = G(B)$.
We have thus,
$$ 
\xymatrix{ & G(M)  \ar[r]^-{\sim} & {\rm End}_B(M) \\
	G(\{\star\}) \ar@{^{(}->}[r] & 
	G(B) \ar@{^{(}->}[r] \ar[u] & 
	{\rm Aut}_B(M)  \ar[u] 
}
$$
where the maps in the lower row are injective group morphisms.

\begin{definition}\label{df:Galois_st}
	Let $\pi\colon M\to B$ be an epimorphism in ${\mathcal C}$. A Galois structure for $\pi$ is an splitting action $\alpha\colon G\times M \to M$ for ${\rm KP}_\pi$ such that the induced group morphism
	$$G(\{\star\})\xrightarrow{\sim}{\rm Aut}_B(M),\quad g\mapsto \alpha_g$$
	is an isomorphism.
	% We say that a Galois structure is strong in ${\mathcal C}$ if the action $\alpha$ induces an isomorphism $G(\{\star\})\xrightarrow{\sim}{\rm Aut}_B(M)$.
\end{definition}

\begin{definition}
	We say that an epimorphism $\pi\colon M\to B$ of ${\mathcal C}$ is Galois if satisfies the following conditions:
		\begin{enumerate}
			\item[(i)] it is normal;
			\item[(ii)] it admits a unique (up to isomorphism) Galois structure.
		\end{enumerate}
		We call Galois group of $\pi$ the group object ${\rm Gal}_\pi$ appearing in the unique Galois structure.
\end{definition}

Note that if $\alpha$ is a Galois structure for $\pi$ then we have isomorphisms,
$$G(\{\star\}) \xrightarrow{\sim} G(B) \xrightarrow{\sim} {\rm Aut}_B(M).$$

\smallskip

%If ${\rm KP}_\pi$ splits in $\mathcal C$ then it also splits in the relative category $\mathcal C_B$ of arrows over $B$.
Given an splitting action $\alpha$ for ${\rm KP}_\pi$ as a groupoid object in $\mathcal C$ we may define an splitting action
$$\tilde \alpha \colon  (G\times B)\times_B M \to M, \quad ((g,b),x) \mapsto \alpha(g,x),$$
for ${\rm KP}_{\pi}$ as a groupoid object in $\mathcal C_B$. In some cases, a splitting action may fail to be a Galois structure in the category  ${\mathcal C}$ but be so in the relative category ${\mathcal C}_B$ of arrows over $B$.

\begin{example} Let ${\bf Set}$ be the category of sets and $\pi\colon M\to B$ be a surjective map. In any case ${\rm KP}_{\pi}$ splits in ${\bf Set}_B$, and any splittig action is a Galois structure. The group object acting is a family of groups indexed by $B$ and acting freely and transitively on the fibers of $\pi$. It is Galois if and only if the fibers have $1$, $2$ or $3$ points.
	
However, ${\rm KP}_\pi$ splits in ${\bf Set}$ if and only if all fibers of $\pi$ have exactly the same cardinal. Finally, $\pi$ is Galois in ${\bf Set}$ if and only if it is a bijection, otherwise we may have the uniqueness for the Galois structure, but $G \subsetneqq {\rm Aut}_B(M)$.
\end{example}

\begin{example}
	Let ${\bf Mnf}$ be the category of smooth manifolds with smooth maps. By direct examination of the definition we have that a an splitting action for a submersion $\pi\colon M\to B$ is an structure of a principal bundle for some structure Lie group $G$. The splitting actions is far from being unique, moreover, $G$ represents a very small part of ${\rm Aut}_B(M)$. 
\end{example}

\subsection{Galois correspondence}

Let us recall that a congruence (internal equivalence relation) in $M$ is a subobject of $M\times M$ having the reflexive, symmetric and transitive property. We say that a congruence $R\subseteq M\times M$ is \emph{effective} if it is the kernel pair of an effective epimorphism. The class of an effective epimorphism up to isomorphisms of the codomain is called an effective quotient. We have then a diagram:
$$R\rightrightarrows M \to M/R.$$

The class  ${\rm Rel}(M)$ of effective of congruences in $M$ is partially ordered. For two congruences represented by monomorphisms $i\colon R\hookrightarrow M\times M$ and $i'\colon R'\hookrightarrow M\times M$ we say that $R\leq R'$ if there is $j\colon R \hookrightarrow R'$ such that $i' \circ j = i$.
Analogously the class ${\rm Quot}(M)$ of effective quotients of $M$ is ordered. For two effective quotients represented by effective epimorphisms $q\colon M\to Z$ and $q'\colon M\to Z'$ we say $q\geq q'$ if $q'$ is $q$-compatible, so that there is $p\colon Z\to Z'$
such that $p\circ q = q'$. There is a natural bijective Galois connection between ${\rm Rel}(M)$ and ${\rm Quot}(M)$ of effective quotients of $M$ given by the adjunctions:
$${\rm KP}\colon {\rm Quot}(M) \to {\rm Rel}(M),\quad (q\colon M\to Z) \mapsto {\rm KP}_q = M\times_Z M,$$
$${\rm coeq}\colon {\rm Rel}(M) \to {\rm Quot}(M),\quad  R \mapsto (q\colon M\to M/R).$$

The quotient by a group action $\alpha\colon G\times M\to M$ is also understood in the above terms. We have $M/G = {\rm coeq}(\alpha,\pi_2)$ if such coequalizer exists in $\mathcal C$. Under suitable assumptions on the existence and nature of quotients by group actions, the general Galois connection gives rise to the classical Galois correspondence.

\begin{theorem}\label{th:GC}
Let $\pi\colon M\to B$ be a Galois epimorphism. Let us assume the following:
\begin{itemize}
    \item[(a)] any subgrupoid object of the action groupoid ${\rm Gal}_{\pi}\ltimes M$ is of the form $H\ltimes M$ where $H$ is a subgroup object of ${\rm Gal}_\pi$;
    \item[(b)] for any subgroup object $H\subseteq {\rm Gal}_\pi$ it does exists the effective quotient $M/H$.
\end{itemize}
Then the following sentences hold:
\begin{itemize}
    \item[(i)] The assignation:
$$H\subseteq G  \,\,\, \leadsto \,\,\, q_H \colon M \to M/H,$$
establishes an order reversing bijective correspondence between the partially ordered class ${\rm Sub}(G)$ of subgroup objects of $G$ and the partially ordered class  ${\rm Quot}_{\geq \pi}(M)$ of intermediate effective quotients of $M$.
    \item[(ii)] Let us consider any effective intermediate quotient $q\colon M\to Z$ with corresponding subgroup $H\subseteq {\rm Gal}_\pi$. The restriction of the Galois structure $\alpha$ to $H\times M$ is a Galois structure for $q$.
%    \item[(iii)] Let us assume that $H$ is a normal subgroup object of ${\rm Gal}_{\pi}$ such that there exist the quotient group object ${\rm Gal}_{\pi}/H$. Let us consider the epimorphism $p\colon Z\to B$ such that $\pi = p\circ q$. There is a Galois structure $\bar\alpha$ in $p$ that makes commutative the following diagram:
%    $$    \xymatrix{{\rm Gal}_{\pi}\times M \ar[r]^-{\alpha} \ar[d]^-{\langle \rho, p\rangle}  & M \ar[d]^-{p} \\ ({\rm Gal}_{\pi}/H) \times Z \ar[r]^-{\bar\alpha} & Z} $$ 
\end{itemize}

\end{theorem}

\begin{proof}
(i)
It is clear that the assignation reverses order, for $H\subseteq H'$ we have $q_H \geq q_H'$. In order to see that it is bijective, let us construct its inverse correspondence. Let $q\colon M\to Z$ be a representative of an effective quotient with $q\geq \pi$. The kernel pair ${\rm KP}_q$ is an effective congruence in $M$ and ${\rm KP}_q\leq {\rm KP}_\pi$. The splitting isomorphim establishes an isomorphism of ${\rm KP}_q$ with a subgroupoid object of ${\rm Gal}_{\pi}\ltimes M$ which, by condition (a), is of the form $H_q\ltimes M$ for a subgroup object $H_q$ depending on $q$. We have that the effective epimorphism $q\colon M\to Z$ is equivalent to $q_{H_q}\colon M\to M/H_q$. Then we have:
$$H \leadsto q_H \leadsto H, \quad q \leadsto H_q \leadsto q.$$
(ii) It is enough to note that the splitting isomorphism $\langle \pi_2,\alpha\rangle$ maps $H\ltimes M$ onto ${\rm KP}_q$.
\end{proof}

\section{Classical Galois theory}

\subsection{Covering spaces}

Let ${\bf Top}$ be the category of topological spaces. A covering map $\pi\colon M\to B$, with $M$ and $B$ connected, is a \emph{Galois cover} if $\pi\times {\rm Id}_M\colon M\times_B M \to M$ is a trivial covering space. There is a Galois theory for covering spaces, analogous to classical Galois theory (see, for instance \cite{Khovanskii2014}). 

\begin{theorem}
	Let $\pi\colon M\to B$ be a surjective local homeomorphism with $M$ and $B$ connected. 
	The following are equivalent:
	\begin{enumerate}
		\item[(a)] $\pi$ is a Galois cover.
		\item[(b)] $\pi$ is a Galois in ${\bf Top}$.
		\item[(c)] ${\rm KP}_\pi$ splits in  ${\bf Top}$.
		%\item[(c)] $\pi$ is weakly Galois in ${\bf Top}_B$
		%\item[(d)] $\pi$ is a cover and weakly pre-Galois in ${\bf Top}$.
	\end{enumerate}
	In any case, the Galois group object is ${\rm Gal}_{\pi} = {\rm Aut}_B(M)$ with the discrete topology. 
\end{theorem}

\begin{proof}
	 $(c)\Rightarrow(a)$. Let us assume that there is an splitting isomorphism $\varphi\colon G\times M \xrightarrow{\sim} M\times_B M$. Then we have that the projection on the second factor $G\times M\to M$ is a local homeomorphism. Thus, $G$ is discrete and $M\times_B  M\to M$ is a trivial cover. It follows that $\pi$ is a Galois cover. We also have $(b)\Rightarrow(c)$. 
	 
	Let us see $(a)\Rightarrow(b)$. We assume that $M\times_B M\to M$ is a trivial cover, thus there is a trivialization,
	$$\xymatrix{G\times M \ar[rr]_-{\sim}^-{\varphi} \ar[rd]_-{\bar\pi} & & M \times_B M  \ar[dl]_-{\pi_1}\\ & M}$$
	with $G$ a discrete topological space. Let us check that there is a group structure on $G$ such that it is isomorphic to ${\rm Aut}_B(M)$ and $\varphi$ is the action of ${\rm Aut}_B(M)$ in $M$.
	
	For each $g\in G$ let us consider the map $\sigma(g)\colon M\to M$ defined by the formula $\sigma(g)(x) = \pi_2(\varphi(g,x))$. It is a continuous map that induces the identity on $B$ and thus, an automorphism of $M$ over $B$. On the other hand, let $\sigma$ be an automorphism of $M$ over $B$. Then, the map $x\mapsto \varphi^{-1}(x,\sigma(x))$ is a section of $\bar\pi$. Since $\bar\pi$ is trivial, then there is a unique $g$ in $G$ such that $\varphi^{-1}(x,\sigma(x)) = (g,x)$. We define this $g$ to be $g(\sigma)$. It is easy to check that those bijections inverse of each other. With the group operation in $G$ induced by $\sigma_{gh} = \sigma_g\circ \sigma_h$ then we have that $\varphi$ is a splitting morphism and thus $\pi$ admits a Galois structure, where the action of $G$ in $M$ is isomorphic to that of ${\rm Aut}_B(M)$ endowed with the discrete topology, and thus unique. 
	
	Let us discuss the normality of $\pi$. In this context, it means that the action of ${\rm Aut}_B(M)$ is transitive on the fibers. Let $m_1$, $m_2$ be two points of $M$ in the same fiber. Let $g$ be the element of $G$ such that
	$\varphi(g,m_1) = (m_1,m_2)$. Then, it is clear that $\sigma(g)(m_1) = m_2$. 
\end{proof}

Note that Galois covers are under the hypothesis of Theorem \ref{th:GC}. The subgroupoids of $G\ltimes M$ are of the form $H\ltimes M$ with $H$ a subgroup of $G$ and the quotient $M/H$ exists in $\bf Top$. We obtain the well known correspondence between intermediate coverings and subgroups of $G$. 

%Let us note that a Galois cover $\pi\colon M\to B$ admits a unique splitting action in ${\bf Top}$, but the uniqueness is lost in the relative category ${\bf Top}_B$. Let us consider the case in which $B$ is a connected, locally arc-connected and locally simply connected with non-commutative fundamental group. Let us choose a point $b_0\in B$. We have two non-isomorphic group bundles:
%$$\Pi_{1}(B,b_0)\times B \to B, \quad \Pi_{1}(B,\_)\to B.$$ 
%where the fiber on the second on $b\in B$ is $\Pi_1(B,b)$. Let us consider the classical construction of the universal cover $\pi\colon \tilde B\to B$ such that the elements of $\tilde B$ are homotopy classes of curves starting at $b_0$. It is a Galois cover, but there are at least two non-isomorphic splitting actions in ${\bf Top}_B$. We may consider:
%$$(\Pi_{1}(B,b_0)\times B) \times_B \tilde B \to \tilde B\times_B\tilde B, \quad (([\sigma], b), [\gamma]) \mapsto ([\gamma], [\sigma][\gamma])),$$
%or,
%$$\Pi_1(B,\_) \times_B \tilde B \to \tilde B, \quad ([\tau],[\gamma]) \mapsto ([\gamma], [\gamma][\tau]^{-1}),$$
%two different splitting actions for the universal cover $\pi\colon \tilde B \to B$ in ${\bf Top}_B$.

\subsection{Algebraic Galois extensions}

Let ${\bf Cmm}$ be the category of commutative rings with unit. The dual category ${\bf Cmm}^{\rm op}$ is the category of affine schemes. 

\smallskip

Let us consider an extension of rings $i\colon K \hookrightarrow L$. The dual map \linebreak $i^*\colon {\rm Spec}(L)\to {\rm Spec}(K)$ is an epimorphism in ${\bf Cmm}^{op}$. In this case the kernel pair is ${\rm Spec}(L\otimes_K L) \rightrightarrows {\rm Spec}(L)$ where the source and target maps are the dual of the canonical embeddings $a\mapsto a\otimes 1$ and $a\mapsto 1\otimes a$ respectively. 

\smallskip

Group objects is ${\bf Cmm}^{\rm op}$ are commutative Hopf algebras. Thus, splitting actions in ${\bf Cmm}^{\rm op}$ are the already known Hopf-Galois structures, in the sense of Chase and Sweedler \cite{Chase1969}. It is well known that Hopf-Galois structures are not unique in general. %Thus, a weakly Galois extension may be non-strongly Galois.

\smallskip

Let us revisit classical Galois theory. Let us consider $i$ to be a finite extension of fields. Classically, it is called a Galois extension 
if it satisfies one of the following equivalent conditions (see \cite{SanchoSancho2013} pp. 140-141):
\begin{enumerate}
	\item[(a)] $L$ is separable and normal\footnote{
		It is clear that our categorical definition of normality coincides, in this context, with the classical definition $L^{{\rm Aut}_K(L)} = K$.
	} 
	over $K$.
	\item[(b)] $|{\rm Aut}_K(L)| = {\rm dim}_K L$.
	\item[(c)] $L\otimes_K L$ (with $L$-algebra structure given by the embedding $a\mapsto a\otimes 1$) is a finite trivial\footnote{A finite trivial $L$-algebra is an $L$-algebra isomorphic to a direct product of a finite number of copies of $L$, $\prod_{i\in I}L$.} $L$-algebra.
\end{enumerate}

Let us consider $i\colon K \hookrightarrow L$ a Galois extension, and let $G$ be ${\rm Aut}_K(L)$. Then, it is well known that the trivialization of $L\otimes_K L$ can be realized as a split. 
We have the trivial finite $L$-algebra ${\rm Maps}(G,L)$ and an isomorphism:
$$\varphi \colon L\otimes_K L \xrightarrow{\sim} {\rm Maps}(G,L)= \prod_{g\in G} L , \quad 
a\otimes b \mapsto f_{a\otimes b},$$
where $f_{a\otimes b}(g) = g(a)b$. Now we have that 
${\rm Maps}(G,L) = {\rm Maps}(G,K)\otimes_K L$. Thus, in the dual category we have that the map,
$$\varphi^* \colon {\rm Spec}({\rm Maps}(G,K)) \times_K {\rm Spec}(L) \xrightarrow{\sim} {\rm KP}_{i^*},$$
is a splitting isomorphism of the groupoid ${\rm KP}_{i^*}.$ Noting that
${\rm Maps}(G,K) = {\rm Maps}(G,\mathbb Z)\otimes_{\mathbb Z}K$ we see that the splitting isomorphism can be defined in the category ${\bf Cmm}^{op}$ and not only in the relative category ${\bf Cmm}_K^{op}$. We may state the following result.

\begin{proposition}
	Let us consider $i\colon K \hookrightarrow L$ a finite separable field extension, and $i^*\colon {\rm Spec}(L)\to {\rm Spec(K)}$ its dual morphism. The following are equivalent:
	\begin{enumerate}
		\item[(a)] $i\colon K\hookrightarrow L$ is a Galois extension.
		\item[(b)] $i^*$  is Galois in ${\bf Cmm}^{\rm op}$.
		\item[(c)] $i^*$  is Galois in ${\bf Cmm}^{\rm op}_K$.
	\end{enumerate}
	In such a case, if $G = {\rm Aut}_K(L)$, there is a natural action of $G$ in $L\otimes_K L$ such that $(L\otimes_K L)^G$ is a Hopf $K$-algebra canonically isomorphic to $
	{\rm Maps}(G,K)$.
\end{proposition}

Let us fix a Galois extension $i\colon K \hookrightarrow L$ with group $G$. Let $H$ be a subgroup of $G$. Then, we realize the field of invariants $L^H$ as the equalizer,
$L^H \to L \rightrightarrows L\otimes_{L^H} L.$
Therefore, in the dual category ${\rm Spec}(L^H)$ appears as the effective quotient of ${\rm Spec}(L)$ by the action of the group object $H$. Moreover, since $G\ltimes {\rm Spec}(L)$ is the spectrum of a $L$-trivial algebra, we have that any subgroupoid is of the form $H\ltimes {\rm Spec}(L)$. We are under the hypothesis of Theorem \ref{th:GC}, which in this particular case gives the classical Galois correspondence between intermediate field extensions and subgroups.

\section{Foliated manifolds}

\subsection{Smooth foliated manifolds}

Let $\bf FMn$ be the category of smooth manifolds endowed with regular foliations.
Objects are pairs $(M,\mathcal D)$ where $M$ is a smooth manifold and $\mathcal D$ is an involutive linear subbundle of $TM$. Morphisms $f\colon (M,\mathcal D)\to (M',\mathcal D)$ are smooth maps $f\colon M\to M'$ such that for all $p\in M$ the differential $d_pf$ induces a linear epimorphism from $\mathcal D_p$ to $\mathcal D'_p$. This implies that $f$ maps leaves of $\mathcal D$ onto leaves of $\mathcal D'$ by local submersions.
A manifold $B$ admits two trivial structures of foliated manifold $(B,TB)$, with only a leaf $B$ and $(B,0_B)$ with point leaves.

\smallskip

Let $(G,\mathcal D_G)$ be a group object in $\bf FMn$. It is clear that $G$ is a Lie group. The existence of the identity element implies that the map,
$$(\{\star\}, 0_\star) \to (G,\mathcal D_G), \quad \star \mapsto e,$$
is a morphism of foliated manifolds, so that ${\rm rank}(\mathcal D_G) \leq {\rm rank}(0_\star) = 0$. It follows $\mathcal D_G = 0_G$. By abuse of notation we write $G$ instead of $(G,0_G)$. If is also clear that an action of $G$ in $(M,\mathcal D)$ to the  category of foliated manifolds is an cation of $G$ in $M$ by symmetries of $\mathcal D$. That is, for any $p\in M$ and $g\in G$ $d_pL_g(\mathcal D_p) = \mathcal D_{gp}$.

\smallskip

A \emph{flat Ehresmann connection} in a submersion $\pi\colon M\to B$ is an involutive subbundle $\mathcal F\subset TM$ such that for each $p\in M$ the differential $d_p\pi$ is an isomorphism of $\mathcal F_p$ with $T_{\pi(p)}B$. We say that a foliated manifold $(M,\mathcal F)$ is \emph{irreducible} if it contains a dense leaf. Let us first analyze the case in which the basis $M$ has a trivial structure of foliated manifold. 

\begin{proposition}\label{pr:CN}
Let  $\pi\colon (M,\mathcal F)\to (B, TB)$ be an epimorphism of foliated manifolds with ${\rm rank}(\mathcal F) = {\rm dim}(B)$. %where $\mathcal L$ is transversal to the fibers
Then $\pi$ is a submersion and $\mathcal F$ is a flat Ehresmann connection.
\end{proposition}

\begin{proof}
For all $p\in M$ we have that $d_p\pi$ maps $\mathcal F_p$ onto $T_pB$. Therefore $d_p\pi$ is surjective for all $p\in M$ and $\pi$ is a submersion. It is clear that $\mathcal F$ is a flat Ehresmann connection. 
\end{proof}

\begin{proposition}\label{PR:principal}
	Let $\pi\colon (M,\mathcal F)\to (B, TB)$ be a epimorphism of irreducible foliated manifolds with ${\rm rank}\,\mathcal F = {\rm dim}\,B$.
	%where $(M,\mathcal L)$ is irreducible. %and $\mathcal L$ is transversal to the fibers of $\pi$ (i.e. $\mathcal L \cap \ker(d\pi) = 0$). 
	The following are equivalent.
	\begin{enumerate}
		\item[(a)] ${\rm KP}_\pi$ splits in $\bf FMn$.
		\item[(b)] $\pi$ is Galois in $\bf FMn$.
		\item[(c)] There is a Lie group $G$ acting on $M$ such that $\pi$ is a principal $G$-bundle and $\mathcal L$ is a $G$-invariant connection.
		\item[(d)] The above, with a unique $G$.
	\end{enumerate}
	In such a case $G$ is ${\rm Aut}_{(B,TB)}(M,\mathcal F)$.
\end{proposition}

\begin{proof}
	Cases (a) and (c) are equivalent from the very definition of splitting action.	It is also clear that (b) and (d) are equivalent. It remains to prove that (c) implies (d). 
	Let us consider two principal structures $\beta\colon M\times H\to M$ and $\alpha\colon M\times G\to M$ such that $\mathcal F$ is simultaneously $G$ and $H$-invariant. Let us see that these actions are conjugated by a Lie group isomorphism.
	
	Let $\mathcal L$ be a dense leaf in $M$. We consider in $\mathcal L$ its intrinsic structure as smooth manifold, so that the projection $\mathcal L\to B$ is an \'etale map with arc-connected Hausdorff domain. Let us note that $M$ and $B$ are necessarily connected. Let $x$ be any point of $\mathcal L$; there is a unique $h\in H$ such that $\alpha(x,g) = \beta(x,h)$. Let $\mathcal L'$ be the leaf of $\mathcal F$ passing through $\alpha(x,g) = \beta(x,h)$. Let us denote $R_g^{\alpha}$ and $R_h^{\beta}$ the right translations by $g$ and $h$ respectively. Then, $R^\alpha_{g}|_{\mathcal L}$ and $R^{\beta}_h|_{\mathcal L}$ are homeomorphisms of $\mathcal L$ into $\mathcal L'$ that project onto the identity on $B$. They coincide on the point $x$, and thus they are the same,
	$R^\alpha_{g}|_{\mathcal L} =R^{\beta}_h|_{\mathcal L}$. Maps $R^\alpha_g$ and $R^\beta_h$ are smooth and they coincide along the dense subset $\mathcal L$, thus they are equal. Finally, the map 
	$G\to H$ that assigns to each $g$ the only element $h$ such that $\alpha(x,g) = \beta(x,h)$ is a group isomorphism. It is defined by composing and inverting smooth maps, so that, it is a Lie group isomorphism conjugating the actions $\alpha$ and $\beta$.
	
	Moreover, the same argument proves that any automorphism \linebreak $\varphi\in{\rm Aut}_{(B,TB)}(M,\mathcal F)$ must be a translation by an element of $G$. 
\end{proof}

The same idea can be generalized to the case in which the foliated structure of the basis is not trivial, but irreducible. Let $\pi\colon M\to B$ be a manifold submersion, and $\mathcal D$ a foliation in $M$. Let us recall that a flat \emph{$\mathcal D$-connection} (or a flat partial connection in the direction of $\mathcal D$) is a foliation $\mathcal F$ in $M$ that for all $p\in M$ the differential $d_p\pi$ maps $\mathcal F_p$ isomorphically onto $\mathcal D_p$. Note that a flat Ehresmann conection is the same that a flat $TB$-connection.

\smallskip

As in Proposition \ref{pr:CN} if $\pi\colon (M,\mathcal F)\to (B,\mathcal D)$ is an submersion of foliated manifolds with ${\rm rank}\,\mathcal F = {\rm rank}\,\mathcal D$ then $\mathcal F$ is a flat $\mathcal D$-connection.

\begin{theorem}\label{smooth}
	Let $\pi\colon (M,\mathcal F)\to (B,\mathcal D)$ be a epimorphism of irreducible foliated manifolds with ${\rm rank}\,\mathcal F = {\rm rank}\,\mathcal D$.
	The following are equivalent.
	\begin{enumerate}
		\item[(a)] ${\rm KP}_\pi$ splits in $\bf FMn$.
		\item[(b)] $\pi$ is Galois in $\bf FMn$.
		\item[(c)] There is a Lie group $G$ acting on $M$ such that $\pi$ is a principal $G$-bundle and $\mathcal D$ is a $\mathcal D$-partial $G$-invariant connection.
		\item[(d)] The above, with a unique $G$.
	\end{enumerate}
	In such a case $G$ is ${\rm Aut}_{(B,\mathcal D)}(M,\mathcal F)$.
\end{theorem}

\begin{proof} Let us consider $\mathcal L$ a dense leaf of $\mathcal F$. Then $\pi(\mathcal L)$ is a dense leaf of $\mathcal D$. We may proceed as in the proof of Proposition \ref{PR:principal} replacing the role of $B$ by $\pi(\mathcal F)$.
\end{proof}

\subsection{Galois correspondence}

From now on let $\pi\colon (M,\mathcal F)\to (B,\mathcal D)$ be a Galois submersion of irreducible foliated manifolds with ${\rm rank}\,\mathcal F = {\rm rank}\,\mathcal D$ with Galois group $G$. Let us check that we are under the hypothesis of Theorem \ref{th:GC}.

\begin{proposition}
 Any object subgroupoid of the action 
    	groupoid $G\ltimes (M, \mathcal F)$ is of the form $H\ltimes (M,\mathcal F)$ with $H$ a Lie subgroup of $G$.
\end{proposition}

\begin{proof}
    Let $(\mathcal G, \mathcal D')\rightrightarrows (M,\mathcal D)$ be a subgroupoid object of the action groupoid. Then $\mathcal G$ is a Lie subgroupoid of $G\ltimes M$ and if $(g,p)\in \mathcal G$ implies that the $\{g\}\times \mathcal L_p\subset \mathcal G$ where $\mathcal L_p$ is the leaf of $\mathcal F$ through $p$.
    
    Let $\mathcal L$ be a dense leaf of $\mathcal F$. Note that for any $g\in G$ and $p\in M$ the poinf $(g,p)$ is an accumulation point of $\{g\}\times \mathcal L$. Therefore if $(g,p)\in \mathcal G$ implies $\{g\}\times \mathcal L\subseteq \mathcal G$ and therefore $\{g\}\times M\subseteq \mathcal G$. It follows that $\mathcal G = S\times M$ for some submanifold $S\subseteq G$. From the groupoid composition and inversion it follows that $S = H$ a Lie subgroup of $G$.
\end{proof}

By a $G$-manifold we mean a manifold $X$ endowed with a left action of $G$. To any $G$-manifold $X$ it corresponds an associated bundle with fiber $X$,
$$M\times_G X\to B$$
defined as the quotient of the direct product $M\times X$ by the equivalence relation $(pg,x) \sim (p,gx)$ for all $p\in M$, $g\in G$, $x\in X$. The $G$-invariant $\mathcal D$-connection induces an \emph{associated $\mathcal D$-connection} $\mathcal F\times_G 0_X$ which is the projection on $M\times_G X$ of the direct product $\mathcal F\times_G 0_X$. We have that,
$$(M\times_G X, \mathcal F\times_G 0_X) \to (B,\mathcal D)$$
is an epimorphism of foliated manifolds and $\mathcal F\times 0_X$ is a flat Ehresmann $\mathcal D$-connection. In particular if $H$ is a Lie subgroup of $G$ and $X = G/H$ is an homogeneous $G$-space we have,
$$M\times_G(G/H) = M/H$$
and the induced associated $\mathcal D$-connection is just the projection of $\mathcal F$ onto $M/H$. Therefore, in this case, Theorem \ref{th:GC} gives us a Galois correspondence between Lie subgroups of $G$ and associated $\mathcal D$-connections in associated bundles whose fibers are homogeneous $G$-spaces.

\subsection{Galois structures over $(B,\mathcal D)$}

Let us discuss Galois structures in the relative category
${\bf FMn}_{(B,\mathcal D)}$ whose objects are smooth maps of foliated manifolds $(Z,\mathcal D_Z)\to (B,\mathcal D)$. A \emph{group bundle} $G\to B$ is a smooth bundle by Lie groups, where composition, inversion and identity depends smoothly on the base point. A \emph{ group $\mathcal D$-connection} in $G\to B$ is a $\mathcal D$-connection $\mathcal D$ in $G$ such that leaves are compatible with composition. Linear bundles and linear $\mathcal D$-connections are the most usual examples of group bundles and group connections. Group bundles over $B$ endowed with  group $\mathcal D$-connections are group objects in ${\bf FMn}_{(B,\mathcal D)}$. They are the smooth geometric counterpart of differential algebraic groups of finite dimension discussed by Buium in \cite{Buium1992}.

\smallskip

In the case of trivial foliated structure in the basis, group objects are locally Lie groups after change of basis, as the following result explains. 

\begin{proposition}
	Let $B$ be simply connected, and 
	$q\colon (G,\mathcal L)\to B$ a group bundle with group connection (and therefore a group object in ${\bf FMn}_{(B,TB)}$). Let $x$ be a point in $M$
	and $G_x$ the fiber of $G$ over $x$, then  $(G,\mathcal L) \simeq (G_x,\{0\})\times(B,TB)$.
\end{proposition}

\begin{proof}
	The argument is local, so we have to see that for each $x\in B$ there is a neighborhood $U$ of $x$ such that $(G|_U,\mathcal L|_U) \simeq (G_x,\{0\})\times (U,TU)$. If this is the case, for each homotopy class of a path $\gamma$ connecting $x$ and $y$ in $B$ we have a group isomorphism $\gamma_*\colon G_x\to G_y$. If $B$ is simply connected, those homotopy classes are unique for each $y$ and the isomorphisms $\gamma_*$ give us the trivialization of the group connection.
	
	\smallskip
	
	In fact, there are neighborhoods $U$ of $x$ in $B$, $V_x$ of $e_x$ (the identity element) in $G_x$, and $V$ of $e_x$ in $G$, and a decomposition $V \simeq U \times V_x$, such that the horizontal leaves of $\mathcal L$ in $V$ have the form $\{g_x\}\times U$ for fixed $g\in V_x$.
	
	\smallskip
	
	Let us see that, for each $h_x\in G_x$ the leaf $\mathcal F$ of $\mathcal L$ that passes through $h_x$ projects onto $U$. We may also assume that we take $U$ small enough so that each connected component of $G|_U$ contains exactly one connected component of $G_x$. Let $y$ be an accumulation point of $q(\mathcal F)$ inside $U$. Let us consider $h_y$ an element in $G_y$ in the same connected component of $G|_U$ than 
	$h_x$. Then there is a leaf $\mathcal F'$ of $\mathcal L|_U$ passing through $h_x$. Let $U'$ be  $q(\mathcal F')$ which is an open subset that intersects $q(\mathcal F)$. By successive composition of $\mathcal F'$ with the leafs of $\mathcal L$ in $V|_{U'}$ we have that the connected component of $G|_U'$ containing $h_y$ decomposes in leaves of $\mathcal L$. In particular, $\mathcal F\cap G|_{U'}$ is part of a leaf of such a decomposition. Finally, $y\in q(\mathcal F)$. We have seen that $q(\mathcal F)$ is an open subset that contains all its accumulation points inside $U$, so that $q(\mathcal F) = U$. Thus, $G|_U$ decomposes in leaves of $\mathcal L$. 
\end{proof}

For the non-simply connected case, the classification of group connections may follow a similar path to the classification of linear connections. Classes of group connections may be given by classes of representations of the fundamental group $\Pi_1(x,B)$ into the group ${\rm Aut}(G_x)$ of automorphisms of the fiber.
In the case of simply connected $B$ there is no distinction between Galois structures in 
${\bf FMn}$ or in ${\bf FMn}_{(B,TB)}$.

\begin{corollary}
	Let $B$ be simply connected and let $\pi\colon (M,\mathcal L)\to (B,TB)$ be a submersion of foliated manifolds with ${\rm rank}\,\mathcal L = \dim\,B$. Then ${\rm KP}_\pi$ splits in $\bf FMn$ if and only if it splits in ${\bf FMn}_{(B,TB)}$.
\end{corollary}

In the non-simply connected case, non trivial irreducible linear connections give us examples of splitting actions in the relative category. For instance,  we may take, $B = S^1\times S^1$. We take $G = \mathbb R \times B$ and $\mathcal D =  \langle \partial_\theta + u\partial_u, \partial_\phi + \alpha u\partial u \rangle$ where $u$ is the coordinate in $\mathbb R$ and $\alpha$ is an irrational number. Then, we have $(G,\mathcal D)\to (B,TB)$ is a group bundle with an irreducible group connection, locally isomorphic to the trivial additive bundle. The action of $G$ on itself is an splitting action in ${\bf FMn}_{(B,TB)}$.

\subsection{Foliated complex algebraic varieties}

Let $\bf FVar$ be the category of complex regular foliated varieties. Objects are $(M,\mathcal D)$ where $M$ is a complex variety and $\mathcal D$ is an involutive Zariski closed linear subbundle of $TM$. A foliated variety is called \emph{irreducible} if it has a Zariski dense leaf, or equivalently, it does not have rational first integrals (except locally constant functions). Group objects in ${\bf FVar}$ are complex algebraic groups. 

\smallskip

In this category, we can state Galois theory exactly in a way totally analogous to what has been done in ${\bf FMn}$. 

\begin{theorem}\label{analytic}
	Let $\pi\colon (M,\mathcal F)\to (B,\mathcal D)$ be a submersion of irreducible foliated varieties with ${\rm rank}\,\mathcal F = {\rm rank}\,\mathcal D$. The following are equivalent.
	\begin{enumerate}
		\item[(a)] ${\rm KP}_\pi$ splits $\bf FVar$.
		\item[(b)] $\pi$ is Galois in $\bf FVar$.
		\item[(c)] There is an algebraic group $G$ acting on $M$ such that $\pi$ is a principal $G$-bundle and $\mathcal D$ is a $\mathcal D$-partial $G$-invariant connection.
		\item[(d)] The above, with a unique $G$.
	\end{enumerate}
	In such a case $G$ is ${\rm Aut}_{(B,\mathcal D)}(M,\mathcal L)$.
\end{theorem}

\begin{proof}
	Totally analogous to the proofs given in Proposition \ref{PR:principal} and Theorem \ref{smooth}.
\end{proof}

It is interesting to make the connection of this Galois theory with differential algebra. 
%This is a geometric realization (in the complex case) of the theory of strongly normal extensions of Kolchin \cite{Kolchin1973}. 
Let us fix $\pi\colon (M,\mathcal L)\to (B,\mathcal D)$ a Galois submersion of irreducible foliated varieries with Galois group $G$ and ${\rm rank}\,\mathcal F = {\rm rank}\,\mathcal D= r$. Let us note that, by elimination, it is always possible to find a system of commuting rational vector fields $\vec D_1,\ldots,\vec D_r$ that span $\mathcal D$ on the generic point of $B$. Let us fix $\Delta_B = (\vec D_1,\ldots,\vec D_r)$. We have that the field of rational functions $(\mathbb C(B),\Delta_B)$ is a differential field whose field of constants is $\mathbb C.$

\medskip

The $\mathcal D$-connection $\mathcal F$ induce lifts of the rational vector fields $\vec D_j$ to $\mathcal F$-horizontal rational vector fields $\vec F_i$ in $M$ that span $\mathcal F$ on the generic point of $M$. We set $\Delta_M = (\vec F_1,\ldots,\vec F_m)$ so that $(\mathbb C(M),\Delta_M)$ is also a differential field whose field of constants is $\mathbb C$. Since the projection of $\vec F_j$ is $\vec D_j$ we have that $\pi^*\colon (\mathbb C(B),\Delta_B)\hookrightarrow(\mathbb C(M),\Delta_M)$ is an differential field extension. We have the following geometric characterization of strongly normal extensions due to Bialynicki-Virula. 

\begin{proposition}[\cite{Bialynicki1962}, in \cite{Buium1986} p. 18]\label{BB}
Let $(K,\Delta)\hookrightarrow (F,\Delta')$ be a differential field extension with $K$ relatively algebraically closed in $F$ and algebraically closed field of constants $C=K^{\Delta} = F^{\Delta'}$. The following are equivalent:
\begin{enumerate}
    \item It is strongly normal in the sense of Kolchin.
    \item There are a connected algebraic group $G$ over $C$ and a $K$-variety $W$ such that: 
    \begin{enumerate}
        \item $W$ is a principal homogeneous space modeled over $G_K = G \times_C {\rm Spec}(K)$.
        \item The field of rational functions in $W$ is $F$.
        \item The group $G$ acts faithfully on $F$ by differential automorphisms fixing $K$.
    \end{enumerate}
\end{enumerate}
Moreover the pair $(G,W)$ is uniquely determined up to isomorphism and we have $G(C) = {\rm Aut}_{\Delta}(F/K)$. 
\end{proposition}

This geometric characterization immediately yields the following.  

\begin{proposition}
   Let  $\pi\colon (M,\mathcal L)\to (B,\mathcal D)$ a Galois submersion of irreducible foliated varieries with Galois group $G$, and $\Delta_B$, $\Delta_M$ as above. Assume any of the following equivalent hypothesis:
   \begin{enumerate}
       \item $\mathbb C(B)$ is relatively algebraically closed in $\mathbb C(M)$;
       \item $\pi\colon M\to B$ has connected fibers;
       \item $G$ is connected. 
   \end{enumerate}
   The differential field extension:
   $$\pi^*\colon (\mathbb C(B),\Delta_B)\hookrightarrow(\mathbb C(M),\Delta_M)$$  
   is a strongly normal extension in the sense of Kolchin with Galois group $G$.
\end{proposition}

\begin{proof}
   Let us consider: $M_B = M \times_B\,{\rm Spec}\,\mathbb C(B)$ and $G_B = M \times_C\,{\rm Spec}\,\mathbb C(B)$ as $\mathbb C(B)$-varieties. The splitting isomorhism:
   $$M\times_{\mathbb C} G \xrightarrow{\sim} M \times_B M$$
   changes of basis to an isomorphism of $\mathbb C(B)$-varieties,
   $$M_B \times_{\mathbb C(B)} G_B \xrightarrow{\sim} M_B \times_{\mathbb C(B)} M_B$$
   And therefore $M_B$ is a principal homogenous space over $G_B$. The field of rational functions in $M_B$ is also $\mathbb C(M)$. For any $g\in G$ we have a field automorphism,
   $$R_g^*\colon \mathbb C(M)\to \mathbb C(M)$$
   that fixes $\mathbb C(B)$ and the derivations $\vec F_j$ in $\Delta_M$. This gives an inclusion,
   $$G\to {\rm Aut}_{\Delta}(\mathbb C(M)/\mathbb C(B)),\quad g\mapsto R_g^*$$
   and we conclude the Bialynicki-Virula Proposition \ref{BB}.
\end{proof}

\begin{remark}
The applications to differential algebra seem to go further. There have been several generalizations of differential Galois theory theory \cite{Pillay1998, CassidySinger2006} and a geometric characterization of strongly normal extensions \cite{Kovacic2003, Kovacic2006} which is very much in the flavour of Definition \ref{df:Galois_st}. We expect upcoming research clarifying how all those theories relate with the framework proposed here.
\end{remark}

\section*{Acknowledgements}
We acknowledge the support of Universidad Nacional de Colombia, Universidad de Antioquia, Colciencias project ``Estructuras lineales en geometr\'ia y topolog\'ia'' 776-2017 code  57708 (Hermes UN 38300). We thank Eduardo J. Dubuc and Yuri Poveda who gently spent some time discussing the topic with the first author during a pleasant brief encounter in Pereira, Colombia. We appreciate the help of the anonymous referees whose positive criticism allowed us to improve the quality and readability of our original text.
%%%

% List of references

% Cahiers wants the author's address at the end of the paper:

\vspace{5mm}
\noindent
David Bl\'azquez-Sanz \& \\
Juan Felipe Ruiz Castrillon \\
Facultad de Ciencias \\
Universidad Nacional de Colombia - Sede Medell\'{\i}n \\
dblazquezs@unal.edu.co \\
jfruizc@unal.edu.co

\medskip

\noindent
Carlos Alberto Mar\'{\i}n Arango \\
Instituto de Matem\'aticas \\
Universidad de Antioquia \\
calberto.marin@matematicas.udea.edu.co

\end{document}